\documentclass[12pt, twoside]{article}
\usepackage{amsmath,amsthm,amssymb}
\usepackage{times}
\usepackage{enumerate}
\usepackage{extarrows}
\usepackage{color}
\usepackage{tabularx}
\usepackage{multirow}

\usepackage {graphics}
\usepackage{graphicx}
\usepackage{url}
\def\titlerunning#1{\gdef\titrun{#1}}
\makeatletter
\def\author#1{\gdef\autrun{\def\and{\unskip, }#1}\gdef\@author{#1}}
\def\address#1{{\def\and{\\\hspace*{18pt}}\renewcommand{\thefootnote}{}%
		\footnote {#1}}%
	\markboth{\autrun}{\titrun}}
\makeatother
\def\email#1{e-mail: #1}
\def\subjclass#1{{\renewcommand{\thefootnote}{}%
		\footnote{\emph{Mathematics Subject Classification (2010):} #1}}}
\def\keywords#1{\par\medskip
	\noindent\textbf{Keywords.} #1}

\newtheorem{theorem}{Theorem}[section]
\newtheorem{lemma}[theorem]{Lemma}
\newtheorem{definition}[theorem]{Definition}
\newtheorem{proposition}[theorem]{Proposition}
\newtheorem{remark}[theorem]{Remark}

\newtheorem{example}[theorem]{Example}

\newcommand{\R}{\mathbb{R}}

\numberwithin{equation}{section}

\newcommand{\PreserveBackslash}[1]{\let\temp=\\#1\let\\=\temp}
\newcolumntype{C}[1]{>{\PreserveBackslash\centering}p{#1}}
\newcolumntype{R}[1]{>{\PreserveBackslash\raggedleft}p{#1}}
\newcolumntype{L}[1]{>{\PreserveBackslash\raggedright}p{#1}}

\newcolumntype{I}{!{\vrule width 1pt}}
\newlength\savedwidth

\begin{document}
	
	
	\baselineskip=15pt
	
	
	\titlerunning{   }
	
	\title{Dynamics of globally minimizing orbits in 
	 contact Hamiltonian systems}

\author{Yang Xu \and Jun Yan \and Kai Zhao}

	\date{\today}	
	\maketitle
	\address{Yang Xu: School of Mathematical Sciences, Fudan University, Shanghai 200433, China; \email{xuyang$\_$@fudan.edu.cn}
\and Jun Yan: School of Mathematical Sciences, Fudan University, Shanghai 200433, China;
\email{yanjun@fudan.edu.cn}
\and
Kai Zhao:  School of Mathematical Sciences, Tongji University, Shanghai  200092 , China;
\email{zhaokai93@tongji.edu.cn}
}


	\begin{abstract}
In this paper, we  study the asymptotic behavior of globally minimizing orbits of contact Hamiltonian  systems. Under some assumptions, we prove that the $\omega$-limit set of 
 globally minimizing orbits is contained  in the set of semi-static orbits.
	\end{abstract}


 	\section{Introduction and main results}
 	\subsection{The motivation of this paper}
Let $M$ be a connected, closed and smooth manifold and  $H:T^*M\times \R \to \R$ be a $C^2$ function called a contact Hamiltonian. 
The standard contact form on $T^*M\times \R$ is the 1-form $\alpha=d u-p d x $.  Each $C^2$ function $H(x,p,u)$ determinates a unique vector field $X_H$ defined by the conditions
$$
\mathcal{L}_{X_H}\alpha=- \frac{\partial H}{\partial u} \alpha,\quad \alpha(X_H)=-H,
$$
where $\mathcal{L}_{X_H}$ denotes the Lie derivative along the contact vector field $X_H$.  In local coordinates, the contact vector field $X_H$ generated by $H$  read:
\begin{equation} \label{eq:ode1}
	X_H:\left\{
	\begin{aligned}
		\dot x&=\frac{\partial H }{\partial p}(x,p,u),\\
		\dot p &=-\frac{\partial H }{\partial x}(x,p,u)-\frac{\partial H }{\partial u}(x,p,u) \cdot p,  \\ 
		\dot u&= \frac{\partial H }{\partial p}(x,p,u)\cdot p  -H(x,p,u),
	\end{aligned}
	\right.\quad (x,p,u)\in T^*M \times \R.
\end{equation}

In this paper, we study the $\omega$-limit set of global minimizer. We will reduce the problem of the existence of semi-infinite minimizing orbits to the problem of large time behavior of viscosity solutions. In order to describe our result clearly, we first recall some known work closely related to ours.

 	\begin{itemize}
		\item For classical Hamilton-Jacobi equations (where Hamiltonians are defined on the cotangent bundle $T^*M$ of $M$):
		\begin{itemize}
		\item  If $F(x,p)$ is a $C^2$  Hamiltonian defined on $T^*M$ satisfying Tonelli conditions with respect to the argument $p$.
	In \cite{M,CDI},  the authors prove that
\begin{itemize}
\item[$(\star)$] $\omega$-limit set of an orbit in positive globally minimizers  is contained in the set of static  curves.
\end{itemize}
 
 See \cite{CDI,B1,B3} for details.

   \item If $F(x,p,t)$ is a $C^2$  Hamiltonian defined on $  T^*M \times \mathbb{R}$ satisfying Tonelli condition with respect to the argument $p$. In \cite{B1}, Bernard shows that the $\omega$-limit set of positive globally minimizing orbits  is not equal to the Ma\~n\'e set in some cases. 
		\end{itemize}
		\item For contact Hamilton-Jacobi equations (where Hamiltonians are defined  on $T^* M\times\mathbb{R}$),

		  \begin{itemize}
		  \item   In \cite{SS}, the authors find compact maximal attractor for discounted systems. In \cite{WWY2, JY}, the authors generalize the results of \cite{SS} to contact Hamiltonian system. These works are under assumption 
		  $$ 
		  \frac{\partial H}{\partial u}(x,p,u) >0, \quad \forall (x,p,u)\in T^*M \times \R.
		  $$
		\end{itemize}
	\end{itemize}
 This naturally  leads to a question:
 \begin{center}
 	{\it Can we confirm the conclusion $(\star)$ with contact Hamiltonians under suitable assumptions?}
 \end{center} 
 	 	   This is  the motivation  and purpose of this  paper. 
 	
\medskip

Different from the method of   the Ma\~n\'e action potential in \cite{CDI},  the novelty here is that our research study  the conclusion $(\star)$  in view of  the Large-time behavior of the solution of \eqref{1-1}, see Section 2.

  \subsection{Main results}
 	
 Let $H=H(x,p,u)$ be a $C^3$ function on $T^*M\times \R$ satisfying
 
\begin{itemize}
\item[(H1)]  the Hessian $\frac{\partial^2 H}{\partial p^2} (x,p,u)$ is positive definite for each $(x,p,u)\in T^*M\times\R$;
\item[(H2)]   for each $(x,u)\in M\times\R$, $H(x,p,u)$ is superlinear in $p$;
\item[(H3)]  there is a constant $\kappa >0 $  such that
$$
\Big| \frac{\partial H}{\partial u}(x, p,u)\Big| \leqslant \kappa ,\quad \forall (x,p,u)\in T^*M\times \R.
$$
\end{itemize}

	Consider the  Hamilton-Jacobi equation of the evolution form
	\begin{align}\label{1-1}\tag{$\mathrm{HJ_e}$}
		\partial_t w(x,t)+H(x,\partial_x w(x,t),w(x,t))=0,\quad x\in M,\ t\in[0,+\infty),
	\end{align}
	{and the} stationary form 
	\begin{align}\label{1-2}\tag{$\mathrm{HJ_s}$}
		H(x,\partial_x u(x),u(x))=0,\quad x\in M.
	\end{align}
	Here and anywhere, solutions to equations \eqref{1-1} and \eqref{1-2} should always be understood by the viscosity  solutions, which defined by Crandall and Lions \cite{CL}.

  Let $\mathcal{S}^-$ be the set of viscosity solutions of equation \eqref{1-2}. For any $u_-\in\mathcal{S}^-$, define
	\[
	\Lambda_{u_-}:=\operatorname{cl}\Big(\big\{(x,p,u): x   \in \mathcal{D}_{u_{-}}, p=d_x u_{-}(x),u=u_{-}(x)\big\}\Big),
	\]
	where $\operatorname{cl}(B)$ denotes the closure of $B $ and $\mathcal{D}_{u_-}$ denotes the set of differentiable points of $u_-$. Note that $u_-$ is Lipschitz continuous and $M$ is compact. Thus, $\Lambda_{u_-}$ is a compact subset of $T^*M \times\R$. Moreover, $\Lambda_{u_-}$ is negatively invariant under $\Phi^H_t$, where 
	$\Phi^H_t$ stands for the local flow of the contact Hamiltonian equations \eqref{eq:ode1}. Following \cite{M,WWY2}, we can define positive globally minimizing orbits,  semi-static orbits and Ma\~n\'e set $\widetilde{\mathcal{N}}_{u_-}$ for contact Hamiltonian system \eqref{eq:ode1} by using the implicit action functions introduced in \cite{WWY} , see Definition \ref{def:1} and Definition \ref{def:2} for details.  
Moreover, we assume that
	\begin{itemize}
		\item [(A1)]
				$
	\displaystyle	\frac{\partial H}{\partial u} (x,p,u)>0 , \quad \forall (x,p,u)\in E:=\{ (x,p,u)\in T^*M\times \R :H(x,p,u)=0  \}.$
		\item [(A2)] 
		$
		\displaystyle \frac{\partial H}{\partial u} (x,p,u)\geqslant  0 , \quad \forall (x,p,u)\in T^*M \times \R.  
		$
		\item [(B)] Equation	\eqref{1-2}  admits at least one solution $u_0$.
	\end{itemize}
In particular, if there exist $u_1,u_2\in \R$ satisfying
$\displaystyle \max_{x\in M}	H(x,0,u_1) <0< \min_{x\in M}	H(x,0,u_2),$
then  condition (B) holds \cite{WY2021}.

The following results are considered under the assumptions of conditions (A1) and (B). These conclusions are also valid for conditions (A2) and (B).  
 	\begin{theorem} \label{thm2}
 		For each  orbit $(x(t),p(t),u(t)),t\geqslant 0$ in the set of  (positively) global minimizers, there exists a solution $u_-\in  \mathcal{S}^-$ such that the $\omega$-limit set of $(x(0),p(0),u(0))$  is  contained in Ma\~n\'e set $\widetilde{\mathcal{N}}_{u_-}$, i.e. \
 		$$
 		\omega (x(0),p(0),u(0)) \subset \widetilde{\mathcal{ N}}_{u_-} \subset \Lambda_{u_-}. 
 		$$
 	\end{theorem}
 		\begin{theorem}\label{cor1}	
		For any $(x_0,u_0)\in M\times \R$, there exist  $p_0\in T^*_xM$ and $(x(t), p(t),u(t))$ such that $(x(0), p(0), u(0)) = (x_0,p_0,u_0)$ and there exists a solution $u_-\in \mathcal{S}^-$  satisfying
		$$
		\omega  (x_0,p_0,u_0) \subset \widetilde{\mathcal{N}}_{u_-} \subset \Lambda_{u_-} .
		$$
 		\end{theorem} 
 		\begin{remark}
 			Under condition  (A1) and (B), the solution of \eqref{1-2} is unique by Theorem \ref{thm1}.
 		\end{remark}
 		\medskip
	
	\section{Convergence of solutions}
Our method depends on relationships between the Large-time behavior of the solution and the Large-time dynamics of globally minimizing orbits in  contact Hamiltonian systems. We recall some definitions first.
Under assumptions (H1)-(H3), the authors \cite{WWY1} introduce two semigroups of operators $\{T_t^-\}_{t\geqslant 0}$ and $\{T_t^+\}_{t\geqslant 0}$.  It  has been  proved  in \cite[Theorem 1.1]{WWY1} that the function  $(x,t)\mapsto T^-_t\varphi(x)$ is the unique viscosity solution   of the evolutionary Hamilton-Jacobi equation \eqref{1-1} with	$\omega (x,0)=\varphi(x)$. As $\mathcal{S}^-$ denotes the set of viscosity solutions to \eqref{1-2}, let  $\mathcal{S}^+$ be the set of viscosity solutions to $ H(x,-\partial_x u,-u)=0, x\in M.$ It is well-known 
    that  $u\in\mathcal{S}^-$ if and only if $T^-_tu=u$ for all $t\geqslant 0$; $v\in\mathcal{S}^+$ if and only if $T^+_tv=v$ for all $t\geqslant 0$.
	
		\begin{theorem} \label{thm1}
	  Assume (A1) (B) hold, then the solution of $ H(x,\partial_x u, u)=0$ is unique, which is  denoted by $u_-$. And for any $\varphi\in C(M,\R) $, 
	   $$
	   \lim_{t\to +\infty} T_t^- \varphi(x) =u_-(x) , \quad x\in M.
	   $$
 	\end{theorem}
 		\begin{proposition}\cite{SWY}\label{prop:SWY}
		Assume (A2)(B) hold, then  for any $\varphi \in C(M,\R) $,  there exists $u_-\in \mathcal{S}^- $ such that 
		$$
	 \lim_{t\to +\infty}  T_t^{-} \varphi (x)=u_-(x), \quad x\in M.
		$$   
 		\end{proposition}

	\subsection{Proof of Theorem \ref{thm1}}


The proof of the Theorem \ref{thm1} is divided into the following parts.
\begin{lemma} \label{lem2.1}
	Let $u_-\in \mathcal{S}^-$ be a solution of \eqref{1-2}, then  there exists $\delta_0>0$, such that for any $\varphi\in C(M,\R)$ satisfying $\|\varphi-u_-\|_\infty<\delta_0$, there holds 
	$\displaystyle \lim_{t\to +\infty }T_t^- \varphi =u_-.$
\end{lemma}
\begin{proof}
  By condition (A1), there exist $\kappa_0 >0$  and $\delta_0>0$, such that for  $\theta \in [-1,1]$, 
  $$\frac{\partial H}{\partial u} (x ,\partial_x u_-,u_- - \theta \delta_0) > \kappa_0 >0, \quad 
\forall x \in \mathcal{D}_{u_-}.   $$
 Let $  m = \frac{\kappa_0}{2}$ and
 $\omega_1(x,t):=u_-(x) - \delta_0 e^{-m t} ,  \omega_2(x,t):=u_-(x) + \delta_0 e^{-m t},$
 then for  $\theta \in [0,1] $, 
 \begin{align*}
    \partial_t \omega_1  + H(x, \partial_x \omega_1 ,\omega_1 ) 
     =  & \  \delta_0 m  e^{-m t} + \frac{\partial H}{\partial u} (x,\partial_x u_-,u_- - \theta \delta_0 e^{-m t}) \cdot (- \delta_0 e^{-m t}) <0, \\
       \partial_t \omega_2  + H(x,\partial_x \omega_2 ,\omega_2 )  
     =  & \  - \delta_0 m  e^{-m t} + \frac{\partial H}{\partial u} (x,\partial_x u_-,u_- + \theta \delta_0 e^{-m t} ) \cdot ( \delta_0 e^{-m t}) > 0,
 \end{align*}
 hold for any $ x \in \mathcal{D}_{u_-}$. By the Rademacher's theorem, $\omega_1$ and $\omega_2$ are respectively  almost everywhere subsolution and supersolution of  \eqref{1-1}. Moreover, by (H1) and the semi-concavity of $\omega_2$,  we have $\omega_1$ and $\omega_2$ are respectively the subsolution and supersolution of  \eqref{1-1}. 
Thus, due to $T_t^- \omega_1(x,0), T_t^- \omega_2(x,0)  $ are solutions of  \eqref{1-1}  , applying Proposition \ref{prop:comparison principle},  we have
\begin{align*}
   T_t^- \omega_1(x,0) 
   \geqslant \omega_1(x,t) = u_-(x) - \delta_0 e^{-m t},  \quad
   T_t^- \omega_2(x,0) 
   \leqslant \omega_2(x,t) = u_-(x) + \delta_0 e^{-m t}.
\end{align*}
  For any $\varphi\in  C(M, \mathbb{R}) $  satisfying $ u_-(x) - \delta_0 \leqslant \varphi \leqslant u_-(x) + \delta_0 $, we can get that 
$\omega_1(x,0)  \leqslant  \varphi (x)
\leqslant  \omega_2(x,0).$
It implies that for any $(x,t)\in M \times [0,+\infty)$,
\begin{align*}
	T_t^- \varphi \geqslant     T_t^-\omega_1(x,0) \geqslant u_-(x) -  \delta_0 e^{-m t}, \quad
		T_t^- \varphi \leqslant     T_t^- \omega_2(x,0) \leqslant u_-(x) +  \delta_0 e^{-m t},
\end{align*}
then $\| T_t^-\varphi- u_- \|_\infty < \delta_0 e^{-m t}. $
Thus, we completes the proof.
	\end{proof}

	\begin{lemma} \label{lem2.2}
	Let $u_-\in \mathcal{S}^-$ be a solution of \eqref{1-2}, then for $0 < \delta \leqslant \delta_0$, where $\delta_0$ is in the Lemma \ref{lem2.1}, we have 
		\begin{equation}\label{lem2.2-1}
	\displaystyle \lim_{t\to +\infty} T_t^+ (u_--\delta)=-\infty, \quad \text{uniform} \ x\in M.
		\end{equation} 
	\end{lemma} 
	\begin{proof} Let   $u_1^{\delta} := u_- - \delta $, We claim:  \begin{equation}\label{lem2.2-2}
	\text{
		For any $0 < \delta \leqslant \delta_0$,  $T_t^+u_1^\delta (x)$ is unbounded from below on $M\times (0,+\infty)$.
		}		
	\end{equation} 
	We argue by contradiction and assume that there exists $ \sigma \leqslant \delta_0$ such that $\displaystyle \lim_{t\to +\infty} T_t^+  u_1^{\sigma}>-\infty$.
 Then, set $\displaystyle \lim_{t\to +\infty} T_t^+  u_1^{\sigma} := u_1^+ \in \mathcal{S}^+$, 
	and set
 \begin{equation}\label{020}
     u_1^-= \displaystyle \lim_{t\to +\infty} T_t^-  u_1^{+}\in \mathcal{S}^-.
 \end{equation}
	By condition (A1) and $H(x, \partial_x u_1^-, u_1^- )=0$, using Lemma \ref{lem2.1} we get that there exists  $\delta_1$ such that for any $0<\delta\leqslant \delta_1$, 
  \begin{equation}  \label{01}
      \lim_{t\to +\infty} T_t^- {(u_1^- + \delta )} = u_1^-.
  \end{equation}
	Due to $\displaystyle \lim_{t\to +\infty} T_t^+  u_1^{\sigma} = u_1^+$ , there exists $T_0$, such that $$
	  \|T_t^+  u_1^{\sigma}-   u_1^+ \|_\infty< \delta_1, \quad \forall t\geqslant T_0.
	    $$
	  As $ T_t^+  u_1^{\sigma} \geqslant T_t^+ u_1^+=  u_1^+$ ,  we get  
         $u_1^+	 \leqslant   T_t^+  u_1^{\sigma}  \leqslant u_1^+ +\delta_1   \leqslant   u_1^-+\delta_1$ for any $t\geqslant T_0$.
Thus, 
\begin{equation} \label{001}
  \lim_{ s \to  +\infty } T_s^-  u_1^+	 \leqslant    \lim_{ s \to  +\infty } T_s^-  T_t^+  u_1^{\sigma}  \leqslant   \lim_{ s \to  +\infty } T_s^- (u_1^+ +\delta_1)   \leqslant    \lim_{ s \to  +\infty } T_s^- ( u_1^-+\delta_1 ).
\end{equation}
Combined with \eqref{020} , \eqref{01} and \eqref{001}, we get
     \begin{equation} \label{023}
         \lim_{ s \to  +\infty } T_s^-  T_t^+  u_1^{\sigma}  = u_1^-.
     \end{equation}
	     
     By the way, combined with Lemma \ref{lem2.1}  and Proposition \ref{A6}, we have
	     $$
	 u_-  \geqslant  \lim_{ s \to  +\infty } T_s^- T_t^+  u_-  \geqslant     \lim_{ s \to  +\infty } T_s^- T_t^+  u_1^{\sigma}  \geqslant   \lim_{ s \to  +\infty }  T_{s-t}^- \circ  T_t^- T_t^+ u_1^{\sigma}  \geqslant      \lim_{ s \to  +\infty } T_{s-t}^- u_1^{\sigma} =u_-.
	     $$
	     This implies that
      \begin{equation} \label{024}
          \lim_{ s \to  +\infty } T_s^- T_t^+  u_1^{\sigma}  = u_-.
      \end{equation}
      
	In view of \eqref{023} and \eqref{024}, we get that
 $u_- \equiv u_1^-$. 
 
 \medskip
Next,  we  show a contradiction by proving that
 \begin{equation}\label{eq:neq-empty1}
\{x\in M: u_-(x)=u^+_1(x) \} = \emptyset, \quad    \{x\in M: u^-_1(x)=u^+_1(x) \} \neq \emptyset.
 \end{equation}

On one hand, due to $u_1^\delta$, $\delta \in [0,\delta_0] $ is a subsolution  and  Proposition \ref{A1}, we have
 $$
  T_t^+ u_1^\delta \leqslant u_1^\delta =u_--\delta <u_-, \quad \forall t>0,
 $$
 which implies that
 $$
 u_1^+= \lim_{t\to +\infty}   T_t^+ u_1^\delta \leqslant u_--\delta <u_-.
 $$ 
 Thus, $\{x\in M: u_-(x)=u^+_1(x) \} = \emptyset$.

On the other hand, we show that 
 \begin{equation}\label{eq:neq-empty}
 \{x\in M: u^-_1(x)=u^+_1(x) \} \neq \emptyset.
 \end{equation}
If $u^-_1>u^+_1 $ , then set $\eta:=\displaystyle \min_{x\in M} \{u_1^-(x)- u_1^+(x) \}   $.
   According to \eqref{020} , there is $t_0>0$, such that for $t\geqslant t_0$, 
  $$
   T_t^- u_1^+ \geqslant u_1^--\frac{\eta}{2} > u_1^+.
  $$
  Notice that for any $x\in M, t>t_0$, applying Proposition \ref{A6} and Proposition \ref{A5}, one gets
   $
   u_1^+ \geqslant T_t^+ \circ   T_t^- u_1^+ > T_t^+u_1^+. 
   $
    This contradicts with the fact that $u_1^+\in \mathcal{S}^+$, which is a fixed point of $\{T_t^+\}_{t\geqslant 0} $. 
    Thus, we follows the claim \eqref{lem2.2-2}.  Finally, applying Proposition \ref{A7}, one concludes the result \eqref{lem2.2-1}. 
	\end{proof} 
	\begin{lemma}  \label{lem2.3}
		Let $u_-\in \mathcal{S}^-$ be a solution of $ H(x,\partial_x u, u)=0$,  then for any $\varphi\leqslant u_-$ ,
		$$
		\lim_{t\to +\infty} T_t^-\varphi= u_-.
		$$
	\end{lemma}
	\begin{proof}
		For any given $\varphi\in C(M,\R)$ satisfying  $\varphi\leqslant u_-$,  by Lemma \ref{lem2.2} and Proposition \ref{A7}, there exist $T_1>0$ and $0< \delta \leqslant \delta_0$, where $\delta_0$ is in the Lemma \ref{lem2.1},  such that 
		$$
	\varphi\geqslant  T_{T_1}^+ (u_--\delta).
		$$
		Combined with  Lemma \ref{lem2.1} and  Proposition \ref{A6},  we get
\begin{align*}
    	 u_- & {= \lim_{t\to +\infty}T_t^- u_-} \geqslant\lim_{t\to +\infty}T_t^- \varphi\geqslant 	\lim_{t\to +\infty}T_t^- \circ  T_{T_1}^+ (u_--\delta) \\
     &= 	\lim_{t\to +\infty}T_{t-T_1}^- \circ T^-_{T_1}  \circ T^+_{T_1}  (u_--\delta)\geqslant 	\lim_{t\to +\infty}T_{t-T_1}^-   (u_--\delta)=u_-.  \qedhere
\end{align*} 
	\end{proof}
	
		\begin{lemma} \label{lem2.5}
		 The solution of $ H(x,\partial_x u, u)=0$ is unique. 
	\end{lemma}
	\begin{proof}
Assume that there exist $u_1,u_2\in \mathcal{S}^-$ with {$\|u_1-u_2\|_{\infty}>0$}. Taking $\varphi \leqslant \min\{ u_1,u_2\} $, by Lemma \ref{lem2.3}, we have  
		$$
		\lim_{t\to +\infty} T_t^- \varphi=u_1,\quad 	\lim_{t\to +\infty} T_t^- \varphi=u_2,
		$$
		which    makes a contradiction.
	\end{proof}
	\medskip
		
\noindent{\textbf{Proof of Theorem \ref{thm1}:}} 
First, by Lemma \ref{lem2.1}, Lemma \ref{lem2.3} and Lemma \ref{lem2.5}, $u_-$ is the only viscosity solution to $ H(x,\partial_x u, u)=0$ , and there exists $\delta_0$ such that for any $\varphi\in C(M,\R)$ satisfying  $\varphi \leqslant u_-+\delta_0 $, 
\begin{equation}\label{eq:pf2-1-1}
\lim_{t\to +\infty}T_t^- \varphi=u_-.
\end{equation}
	 Next we focus on the case for any $\varphi\in C(M,\R)$. Obviously,  we only need to consider the following three cases:
\begin{itemize}
	\item[\textbf{(1)}]  $T_t^- \varphi(x)$ is {bounded} on $M\times\R^+$.
	\end{itemize}
	By Proposition \ref{A8}, $\displaystyle  \liminf_{t\to +\infty} T_t^- \varphi \in \mathcal{S}^- $ and $ \Psi(x):= \limsup \limits_{t\to +\infty} T_t^- \varphi $ is a subsolution.
Due to Lemma \ref{lem2.5},  $ \mathcal{S}^-$ has only one element, then $\displaystyle \liminf_{t\to +\infty} T_t^- \varphi=u_- $. By Proposition \ref{A1}  we can get that  $T_t^+ \Psi(x)\leqslant \Psi(x) $ and   $\Psi(x)\geqslant u_-(x) $. Denote 
$$ 
u_+:=\lim_{t\to +\infty} T_t^+ \Psi(x)\in \mathcal{S}^+. 
$$
Since $\displaystyle \lim_{t\to +\infty}   T_t^- u_+\in \mathcal{S^-}$, then $\displaystyle \lim_{t\to +\infty}   T_t^- u_+=u_-$ and $u_+ \leqslant u_- $. Hence, there exists $T_0$ such that $T_{T_0}^+ \Psi(x) \leqslant u_-+\delta_0.$ Then on the one hand,
\begin{equation} \label{eq:pf-thm21-1}
\lim_{t\to +\infty}T_t^-\circ T_{T_0}^+ \Psi(x)\leqslant \lim_{t\to +\infty}T_t^-(u_-+\delta_0) =u_-(x).
\end{equation}
On the other hand, by  Proposition \ref{A6} and \ref{A1}, we have
\begin{equation}\label{eq:pf-thm21-2}
\lim_{t\to +\infty}T_t^-\circ T_{T_0}^+ \Psi(x)   \geqslant \lim_{t\to +\infty}T_{t-T_0}^-\circ T_{T_0}^- \circ T_{T_0}^+ \Psi(x) \geqslant \lim_{t\to +\infty}{T^-_{t-T_0}}\Psi(x) \geqslant \Psi(x)   \geqslant u_-(x).
\end{equation}
Combining with \eqref{eq:pf-thm21-1} and \eqref{eq:pf-thm21-2},  $\Psi(x)=u_-(x)$. Then the conclusion is true in this case.
\medskip
 \begin{itemize}
\item[\textbf{(2)}] If $T_t^- \varphi(x)$ is unbounded   from below on $M\times \R$, this case will not happen.
\end{itemize}

  It is clear that there exists $c=\| \varphi- u_- \|_\infty \in \R^+$, such that $  \varphi \geqslant u_-- c $.  Then by Proposition \ref{A5} and applying Lemma \ref{lem2.3}, we have  
 $$
  T^-_t \varphi\geqslant   T^-_t (u_--c) \quad \text{and} \quad \lim_{t\to +\infty}  T^-_t (u_--c)=u_-.
 $$
 Thus $T_t^- \varphi(x)$ is bounded from below on $M\times \R$. This contradicts the assumption. 

 \begin{itemize}
\item[\textbf{(3)}] If $T_t^- \varphi(x)$ is unbounded  from above on $M\times \R$, this case does not happen either.  
\end{itemize}
Due to Proposition \ref{A7}, for fixed $c \geqslant 2 || u_+ - \varphi||_\infty$, there exists $t_c>0$ such that
$$
T_{t_c}^- \varphi  \geqslant \varphi +c  > u_+,
$$
where  $u_+:=\lim \limits_{t\to +\infty} T_t^+ u_-\in \mathcal{S}^+ $.
Then,  by  Proposition \ref{A5} and \ref{A6} ,  we have
$$
\varphi \geqslant T_{t_c}^+\circ T_{t_c}^- \varphi  \geqslant T_{t_c}^+ (\varphi +c)> T_{t_c}^+ u_+= u_+.
 $$
 It implies that for any $s\in [0,t_c]$,
$$
 u_+ = T^+_{ t_c+s} u_+ \leqslant  T_{nt_c+s}^+ \varphi(x) \leqslant \max_{s\in [0,t_c]} T_s^+ \varphi(x), \quad  \forall n \in \mathbb{Z}^+.
$$
Thus, $\{ T_{t}^+ \varphi(x) \}$ is bounded  on $M \times \R^+$. Hence, by Proposition \ref{A8},  there exists
$$
 \limsup_{t\to +\infty} T^+_t \varphi(x)\in \mathcal{S}^+.
$$
  By Lemma \ref{lem2.5} and  Proposition \ref{A1}, it implies that  
  $$
  \displaystyle u_-=  \lim_{s\to +\infty } T_s^- \Big( \limsup_{t\to +\infty} T^+_t \varphi \Big)\geqslant  \limsup_{t\to +\infty} T^+_t \varphi,$$
   and thus there exists $T_1$ such that $T^+_{T_1} \varphi \leqslant u_-+\delta_0 .$
Then  on the one  hand, \eqref{eq:pf2-1-1} implies 
\begin{equation}\label{eq:2-1-pf-1}
\limsup_{t\to +\infty}T_t^-\circ T_{T_1}^+ \varphi \leqslant \lim_{t\to +\infty}T_t^-(u_-+\delta_0)=u_-.
\end{equation}
On the other  hand, for any $n\in \mathbb{N}$, by Proposition \ref{A6},
\begin{equation}\label{eq:2-1-pf-2}
	\lim_{n\to +\infty}T_{nt_c+T_1}^-\circ T_{T_1}^+ \varphi (x) \geqslant \lim_{n\to +\infty}T_{nt_c}^-\varphi (x) \geqslant \varphi (x)+ c   > u_+(x) + c.
\end{equation}
In view of \eqref{eq:2-1-pf-1} and \eqref{eq:2-1-pf-2}, we get $u_- >u_++c $. However, it is quite similar with \eqref{eq:neq-empty}  to show   that     $\{ x\in M: u_-(x)= u_+(x) \} \neq \emptyset$.
This makes  a contradiction, which completes the proof. \qed 
  \medskip
	
	\section{Dynamic of globally minimizing orbits}
	The authors of \cite{WWY}  provide the implicit variational principle for contact Hamiltonian systems and introduce the notion of implicit action functions $h_{x_0,u_0}(x,t),h^{x_0,u_0}(x,t)$. The basic properties of implicit action functions see \cite{WWY,WWY1}.
	
		\begin{definition}[\cite{WWY2}, {Definition 3.1}] \label{def:1}
		A curve $(x(\cdot),u(\cdot)):\mathbf{R}\to M\times\mathbf{R}$ is called \textbf{globally minimizing}, if it is locally Lipschitz continuous and
		for each $t_1$, $t_2\in\mathbf{R}$ with $t_1< t_2$, there holds
		\begin{align}\label{3-1}
			u(t_2)=h_{x(t_1),u(t_1)}(x(t_2),t_2-t_1).
		\end{align}
		Moreover, a curve $(x(\cdot),u(\cdot)):\mathbf{R}^+ \to M\times\mathbf{R}$ is called \textbf{positive globally  minimizing}, if it is locally Lipschitz and \eqref{3-1} holds
		for each $t_1$, $t_2\in\mathbf{R}^+$ with $t_1< t_2$. 
	   \end{definition}
		\begin{definition}\label{def:2}
		A curve $(x(\cdot),u(\cdot)):\mathbf{R}\to M\times\mathbf{R}$ is called \textbf{semi-static}, if it is globally minimizing and for each $t_1$, $t_2\in\mathbf{R}$ with $t_1\leqslant t_2$, there holds
		\begin{equation}\label{3-3}
			u(t_2)=\inf_{s>0}h_{x(t_1),u(t_1)}(x(t_2),s).
		\end{equation}	

	\end{definition}
	If  the curve $(x(\cdot),u(\cdot)):\mathbf{R}\to M\times\mathbf{R}$ is positive globally  minimizing, then by \cite[Proposition 3.1]{WWY2}, { we know that $(x(t),p(t),u(t))$} is an orbit of $\Phi^H_t$, where 
	$$
	p(t)=\frac{\partial L}{\partial \dot{x}}(x(t),\dot{x}(t),u(t)).
	$$
	 We call it  a positive globally  minimizing  orbit of $\Phi^H_t$.	 Semi-static orbit can be similarly defined.

	\begin{definition} \label{audeine}
		We call the set of all semi-static orbits  the\textbf{  Ma\~n\'e set} of $H$, denoted by $\tilde{\mathcal{N}}$. 
	\end{definition}
In fact, the authors show in \cite[Theorem1]{WWY5} that the Ma\~n\'e set can be  classified according to the set of solution $\mathcal{S}^- $ as follows.
\begin{equation}\label{eq:Mane}
\tilde{\mathcal{N}}= \bigcup\limits_{ u_-\in \mathcal{S}^-} \tilde{\mathcal{N}}_{u_-}, \ {\footnotesize \tilde{\mathcal{N}}_{u_-}=\{ (x,p,u)\in T^*M\times \R : u=u_-(x)=u_+(x), p= d_xu_{\pm}(x)  \},}
\end{equation}
where $u_+=\lim\limits_{t\to +\infty} T_t^+ u_-$.

		\subsection{Proof of Theorem \ref{thm2}}
	
	We divide the proof of Theorem \ref{thm2} into the following two steps.
	
	\begin{itemize}
		\item [ ]\textbf{Step1:} \quad Boundedness of globally minimizing orbits
	\end{itemize}
\begin{proposition}{\cite[Lemma 2.1]{WWY1}} \label{prop3.1}
    For any given $a,b,\delta,T \in \mathbb{R}$ with $a<b,0<\delta<T$, there exists a compact set $\mathcal{K} := \mathcal{K}_{a,b,\delta,T} \subset T^*M \times \mathbb{R}$ such that for any $(x_0,u_0,x,t) \in M \times [a,b] \times M \times [\delta,T]$ and any minimizer $\gamma(s)$ of $h_{x_0,u_0}(x,t)$, we have
    \begin{equation*}
        (\gamma(s),p(s),u(s)) \subset \mathcal{K}, \quad \forall s \in[0,t],
    \end{equation*}
    where $u(s) = h_{x_0,u_0}(\gamma(s),s)$, $p(s)=\frac{\partial L}{\partial v} (\gamma(s),\dot{\gamma}(s),u(s))$ and $\mathcal{K}$ depends only on $a,b,\delta$, and $T$.
\end{proposition}
 
		\begin{lemma}\label{lem-bound-minimizer}
 		For each positively global minimizer
 		$$
 		 (x(t),p(t),u(t)) ,   t\geqslant 0,
 		 $$
 		it is bounded on $[0,+\infty)$.
\end{lemma} 
\begin{proof}
The boundedness of $x(t)$ is obvious. Let us focus on the boundedness of $u(t)$ and $p(t)$ in turn.
Actually, for the case of $t\geqslant s \geqslant 0 $,  from the definition  of positive globally  minimizing  we can get  that $u(t)= h_{x(s),u(s)}(x(t),t-s). $
		By taking $s=0$, we have
		\begin{align*}
			|u(t)|=&\, | h_{x(0),u(0)}(x(t),t)| 
			= \Big|h_{x(0),u(0)}(x(t),t)- h_{x(0),u(0)}(x(0),t)+h_{x(0),u(0)}(x(0),t)\Big| \\
			=&\, |T_{t-1}^- h_{x(0),u(0)}(x(t),1)- T_{t-1}^- h_{x(0),u(0)}(x(0),1)|+ |h_{x(0),u(0)}(x(0),t)|. 
		\end{align*}

         Let $(x(0), p(0), u(0)) = (x_0,p_0,u_0)$.
		As the convergence of $T^-_t$, there exists a constant $K>0$  such that $|T_t^- h_{x_0,u_0}(x,1)|\leqslant K $ for all $x\in M $ and all $t>1$. Combined with  \cite[Lemma 3.1]{WWY1}, there exists a consistent Lipschitz constant $l_1>0$  such that   
	\begin{align*}
		|T_t^- h_{x_0,u_0}(x,1)- T_t^- h_{x_0,u_0}(y,1)| \leqslant &\,  \sup_{z\in M} \Big|h_{z,T_{t-1}^- h_{x_0,u_0}(z,1)  } (x,1)- h_{z,T_{t-1}^- h_{x_0,u_0}(z,1)  } (y,1) \Big| \\
		\leqslant &\, l_1 |x-y|, 
	\end{align*}
which means that $\{T_t^- h_{x_0,u_0}(x,1) \}_{t>2} $ is uniformly bounded and equi-Lipschitz continuous.	Thus, 
		$$
			|u(t)|
			\leqslant {l_1} \cdot |x(t)-x(0)|+ |T_{t-1}^- h_{x(0),u(0)}(x(0),1) |,  
	$$
		so $u(t)$ is bounded.

  Actually, we can get from the definition of positive globally  minimizing that
		$$
		u(t)= h_{x(t-1),u(t-1)}(x(t+1),2) . 
		$$
		Combining $u(t)$ is bounded and  applying Proposition \ref{prop3.1}, $p(t)$ is bounded on $[0,+\infty)$.
	\end{proof}

	\begin{itemize}
		\item [ ]\textbf{Step2:} \quad $\omega$-limit set of globally minimizing orbits $\subset \widetilde{\mathcal{N}}_{u_-} \subset \Lambda_{u_-} $
	\end{itemize}
	
Due to the boundedness of globally minimizing orbits, the $\omega$-limit set $\omega (x(0),p(0),u(0)) $ is not empty.
Take $(\bar{x},\bar{p},\bar{u})\in\omega {(x_0,p_0,u_0)}$. Let $(\bar{x}(t),\bar{p}(t),\bar{u}(t))=\Phi^H_t(\bar{x},\bar{p},\bar{u}),\ t\in\R$. There is a sequence $\{t_n\}\subseteq\R$ such that 
		$$t_n\to+\infty\quad \text{and} \quad (x(t_n),p(t_n),u(t_n))\to(\bar{x},\bar{p},\bar{u}), \quad as\ n\to+\infty.$$
		For any $0\leqslant s<t  $,  from the definition of globally minimizing we get that
		$$u(t_n+t)=h_{x(t_n+s),u(t_n+s)}(x(t_n+t),t-s).$$
		Using the continuity of {$(x_0,u_0,x)\mapsto h_{x_0,u_0}(x,t-s)$} and $\Phi_t^H $,  taking $n\to +\infty $,  we have
		\begin{equation}\label{addeq1}
			\bar{u}(t)=h_{\bar{x}(s),\bar{u}(s)}(\bar{x}(t),t-s),\quad \forall t>s.
		\end{equation}
		Thus,  $(\bar{x}(\cdot ),\bar{u}(\cdot))$ is globally minimizing. 
  
Besides, $h_{x_0,u_0}(x,t+1)=T^-_th_{x_0,u_0}(x,1),$ for any $ x\in M,\ t\geqslant0.$
		Applying Theorem \ref{thm1}, there exists $u_-\in \mathcal{S}^-$ such that
		$$\lim_{t\to+\infty}h_{x_0,u_0}(x,t+1)=\lim_{t\to+\infty}T^-_th_{x_0,u_0}(x,1)=u_-(x), \quad x\in  M.$$
		 Combining  
$\{T_t^- h_{x_0,u_0}(x,1) \}_{t>2} $ is uniformly bounded and equi-Lipschitz, we have
\begin{align*}
&\,  \lim_{n\to+\infty}\Big| h_{x_0,u_0}(x(t_n+t),t_n+t)-u_-(\bar{x}(t)) \Big|\\
 \leqslant &\, \lim_{n\to+\infty} \big( \Big| h_{x_0,u_0}(x(t_n+t),t_n+t)-h_{x_0,u_0}(\bar{x}(t),t_n+t) \Big|+\Big| h_{x_0,u_0}(\bar{x}(t),t_n+t)-u_-(\bar{x}(t)) \Big| \big) \\
 \leqslant &\, \lim_{n\to+\infty} \big( {l_1}\cdot | x(t_n+t)-\bar{x}(t) |+0 \big)=0.
\end{align*}
	Therefore, for any $t\in\R$, we get 
	\begin{equation}\label{eq:pf-cor-1}
		 \bar{u}(t)=\lim_{n\to+\infty}u(t_n+t)= \lim_{n\to+\infty}h_{x_0,u_0}(x(t_n+t),t_n+t)=u_-(\bar{x}(t)).
	\end{equation}
		From \eqref{addeq1}  , we know that for any $t\geqslant s$,  
		\begin{equation*}
			\begin{split}
				&\lim_{\tau\to s}h_{\bar{x}(s),\bar{u}(s)}(\bar{x}(t),t-\tau)=\bar{u}(t)=u_-(\bar{x}(t)) = 
    T^-_{\nu}u_-(\bar{x}(t)) \\ 
    = & \inf_{z\in M}h_{z,u_-(z)}(\bar{x}(t),\nu) \leqslant h_{\bar{x}(s),u_-(\bar{x}(s))}(\bar{x}(t),\nu)=h_{\bar{x}(s),\bar{u}(s)}(\bar{x}(t),\nu),\ \forall \nu >0.
			\end{split}
		\end{equation*} 
		As a consequence, 
		$$\bar{u}(t)=\lim_{\tau\to s}h_{\bar{x}(s),\bar{u}(s)}(\bar{x}(t),t-\tau)=\inf_{\nu>0}h_{\bar{x}(s),\bar{u}(s)}(\bar{x}(t),\nu),\quad \forall t\geqslant s.$$
		So, $(\bar{x}(\cdot),\bar{u}(\cdot))$ is semi-static.
		\medskip
		
		Finally, we show that
		$
		\bar{p}(t)= \frac{\partial L}{\partial v}(\bar x(t), \dot {\bar x}(t), \bar u(t)).
		$
   For $\tau \in \mathbb{R}$,
   set 
   $$ \big(\tilde{x}(\tau),\tilde{p}(\tau) ,\tilde{u}(\tau) \big) = \Phi_{\tau}^H  \big(\bar{x}, \frac{\partial L}{\partial v}(\bar {x}, \dot {\bar {x}}, \bar {u}), \bar{u} \big ) .$$
   It follows that $\bar{x}(0) = \tilde{x}(0) = \bar{x}$, $\bar{u}(0) = \tilde{u}(0) = \bar{u}$. Since $u_-$ is a solution of equation \eqref{1-2}, then  $\big(\tilde{x}(\tau),\tilde{p}(\tau) ,\tilde{u}(\tau) \big) \in \Lambda_{u_-}$.
We argue by contradiction and  assume that $\bar{p} \neq \frac{\partial L}{\partial v}(\bar x, \dot {\bar x}, \bar {u})$.
   Since 
   $$\bar{u} = u_-(\bar{x}) = T^-_1 u_-(\bar{x}) \leqslant h_{\bar{x}(-1), u_-(\bar{x}(-1))}(\bar{x},1) \leqslant \bar{u} ,$$ then 
\begin{align*}
    u_-({\tilde{x}(1)} ) &= T^-_2u_-(\tilde{x}(1)) \leqslant h_{\bar{x}(-1), u_-(\bar{x}(-1))}(\tilde{x}(1),2) < h_{\bar{x}, h_{\bar{x}(-1), u_-(\bar{x}(-1))}(\bar{x},1)}(\tilde{x}(1),1)  \\ &
= h_{\bar{x}, \bar{u}}(\tilde{x}(1),1)   =  h_{\tilde{x}(0), u_-(\tilde{x}(0))}(\tilde{x}(1),1)
\leqslant \tilde {u}(1) = u_-({\tilde{x}(1)} ) ,
\end{align*} 
which makes a contradiction. Hence, we obtain that the $\omega$-limit set of positive globally minimizing orbits is contained in the set of semi-static orbits. Since \eqref{eq:Mane}, one follows
\begin{align*}
    \omega (x(0),p(0),u(0)) \subset \widetilde{\mathcal{N}}_{u_-} \subset \Lambda_{u_-}.  
\end{align*}  
We have completed the whole proof of the Theorem \ref{thm2}. \qed
   
\subsection{Proof of Theorem \ref{cor1}}

We show that there exists $p_0\in T^*_xM$ such that $\Phi_t^H(x_0,p_0,u_0), t>0$ is a positively
global minimizer.
		 {From Theorem \ref{thm1}, we have}
		 $$
		 \lim_{t\to +\infty} h_{x_0,u_0} (x,t)= \lim_{t\to +\infty}T_{t-1}^- h_{x_0,u_0} (x,1)= u_-(x).
		 $$
		 For each $n\in \mathbb{N}$,  by \cite[Theorem 2.1]{WWY2},  $h_{x_0,u_0} (x,n)$ admits a minimizer $\gamma_n$  satisfying 
	\begin{align*}
		 u_n(t)=&\, h_{\gamma_n(s),u_n(s)}(\gamma_n(t),t-s), \quad 0<s<t\leqslant n,\\
		  p_n(t)=&\, \frac{\partial L}{\partial v}( \gamma_n(t), \dot  \gamma_n(t), u(\gamma_n(t)) ), \quad  t>0.
	\end{align*}
		 Taking $n\to +\infty$, as $\{p_n(0)\}$ is bounded, there exists convergent subsequence $\{p_{n_k}(0)\}$, denoted by $p_{n_k}(0) \to p_0 $ as $ k \to + \infty$. 
		 
		 \medskip
For any fixed $0< t_1<t_2$, we have 
\begin{align*}
    & u_{n_k}(t_2)=h_{\gamma_{n_k}(t_1),u_{n_k}(t_1)}(\gamma_{n_k}(t_2),t_2-t_1), \\
    & p_{n_k}(t_1)= \frac{\partial L}{\partial v}( \gamma_{n_k}(t_1), \dot  \gamma_{n_k}(t_1), u(\gamma_{n_k}(t_1)) ) .
\end{align*}
By the solution depending on initial value and the uniqueness of solution, one has  
$$
 \displaystyle \lim_{k \to +\infty}  \gamma_{n_k}(t_i)=x(t_i ) , \quad 
 {\displaystyle \lim_{k \to +\infty} u_{n_k}(t_i)=u(t_i), \quad \displaystyle \lim_{k \to +\infty} p_{n_k}(t_i)=p(t_i) , \quad i=1,2.}
$$
Then, we get 
\begin{align*}
		&	u(t_2)=h_{x(t_1),u(t_1)}(x(t_2),t_2-t_1),  \quad  0< t_1<t_2, \\
   & p(t_1)=\frac{\partial L}{\partial v}( x(t_1), \dot  x(t_1), u(x(t_1)) ), \quad  t_1>0.
		\end{align*}	
By arbitrariness and Definition \ref{def:1},  $\Phi_t^H(x_0,p_0,u_0), t>0$ is a positive globally minimizer. Combined with Theorem \ref{thm2}, we complete the proof. $\hfill\square$

\begin{example}
	We consider the dissipative pendulum model with Hamiltonian $H(x,p,u)= \frac{1}{2} p^2 -1+\cos x +u$ with $x\in [0,2\pi]$. The  Hamiltonian  admits  the Ma\~n\'e set  $\{	(0,0,0 ) \}$, which is a hyperbolic fixed point of \eqref{eq:ode1}. By the Theorem \ref{thm2} and Theorem \ref{cor1},  there exists $ p_0\in T^*M $ such that  $ \omega(x_0,p_0,u_0) = \{(0,0,0 )\}$ 
 if and only if $ (x_0,p_0,u_0)\in W^s(0,0,0 ) $, where $W^s(0,0,0 )$ is denoted by the stable  manifold about point $(0,0,0)$ with respect to \eqref{eq:ode1}. We can  project \eqref{eq:ode1} onto the $x,p$-plane as 
 $$
 \begin{cases}
 	\dot x= p, \\
 	\dot p=\sin x-p, 
 \end{cases}
 $$
  as shown in the figure below.
 \begin{figure}[h]
\begin{center}
\includegraphics[width=11cm]{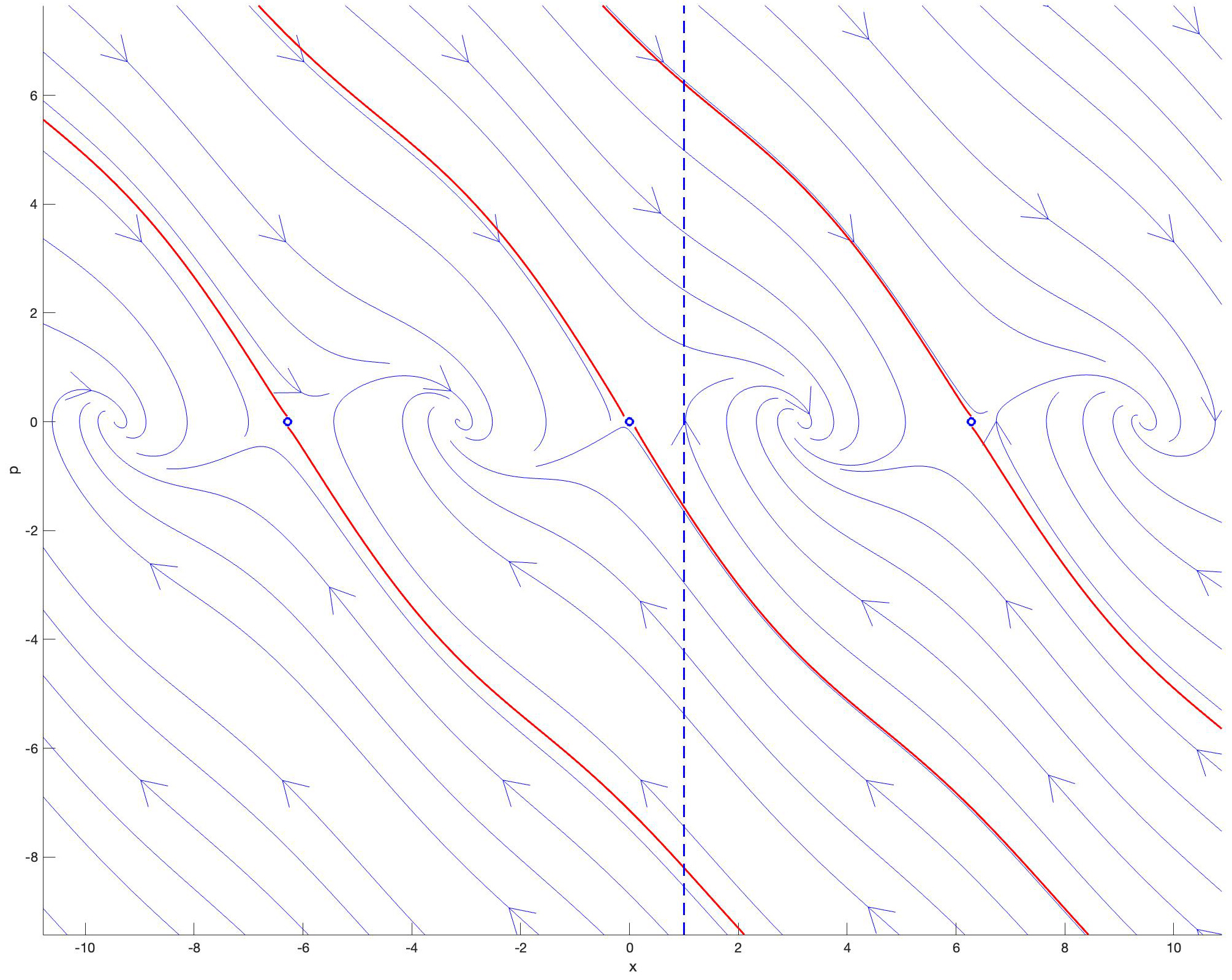}
\end{center}
\caption{The stable 
manifolds of the saddle are highlighted by thick red line, and  the dotted blue line  represents  the line $(x_0,p),p\in \R$. The intersection of both of them  admits the selection of $p_0$.}
\end{figure}

\end{example}

 \medskip

\appendix 
\section{Some propositions for semigroups  $T_t^\pm$}
 
   \begin{proposition} \cite[Proposition 4.3]{WWY1} \label{A5}
 	 Let $\varphi_1,\varphi_2 \in C(M,\R)$, if $\varphi_1(x) < \varphi_2 (x) $ (resp. $\varphi_1(x) \leqslant \varphi_2 (x) $ ) for each $x\in M$, then 
	$$
	T_t^{\pm} \varphi_1(x) < T_t^{\pm} \varphi_2 (x), \Big(resp. \ T_t^{\pm} \varphi_1(x) \leqslant T_t^{\pm} \varphi_2 (x) \Big)  \quad \forall (x,t)\in M \times \R^+.
	$$
 \end{proposition}

   \begin{proposition} \label{A6}
 	\cite[Proposition 10]{W-Y} Let $\varphi \in C(M,\R)$, then
 $$T_t^- \circ T_t^+ \varphi \geqslant \varphi ,\quad T_t^+ \circ T_t^- \varphi \leqslant \varphi,  \quad  \forall \ t >0.$$
 \end{proposition}
\begin{proposition}{ \cite[Thmorem.1.2]{WWY1}}\label{A8}
Let $\varphi\in C(M,\R)$, if the function $(x,t)\mapsto T_t^- \varphi(x) $ or $(x,t)\mapsto  T_t^+ \varphi(x) $  is bounded on $M \times [0,+\infty)$, then 
 $$ \liminf_{t\to +\infty} T^-_t \varphi(x) \in \mathcal{S}^- , \quad  \limsup_{t\to +\infty} T^+_t \varphi(x) \in \mathcal{S}^+,
  $$  
  and $\displaystyle \limsup_{t\to +\infty} T^\pm_t \varphi(x) $ are subsolutions.
\end{proposition}

\begin{proposition} \cite[Proposition 16]{SYZ} \label{A7}
Let $\varphi \in C(M,\mathbb{R})$.  If $T_t^-\varphi(x)$ is unbounded from above on $M\times (0,+\infty)$, then for any $c\in \mathbb{R}$, there is $t_c>0$ such that $T_{t_c}^-\varphi(x)\geqslant \varphi(x)+c,  $ for any $ x\in M$. 
If $T_t^+\varphi(x)$ is unbounded from below on $M\times (0,+\infty)$, then for any $c\in \mathbb{R}$, there is $t_c>0$ such that $T_{t_c}^+\varphi(x)\leqslant \varphi(x)-c $ for any $  x\in M$.	
 
\end{proposition}


Recall the comparison principle of Hamilton-Jacobi equation.
\begin{proposition}\label{prop:comparison principle} \cite{L}\cite{IWWY}
	For any given $T>0$, let $v,w\in C(M\times [0, T) ,\R)$  be respectively, subsolution and supersolution of 
	$$
 \partial_x	u+H(x,\partial_x u,u)=0, \quad \forall (x,t)\in M\times (0,T).
	$$ 
	If $w(x,0)\geqslant v(x,0)$ for any $x\in M$, then $w \geqslant v$ on $ M\times [0,T)$.
\end{proposition}
As a consequence, the following Proposition can be obtained.
 \begin{proposition} \label{A1}
 	For any subsolution $\varphi$ of \eqref{1-2}, $T_t^- \varphi\geqslant \varphi $ and $   T_t^+ \varphi\leqslant \varphi $ hold for any $  t >0$.
 \end{proposition}

\section*{Acknowledgements} 
 The second author was supported by National Natural Science Foundation of China (Grant No. 12171096,12231010). The third author was supported by   National Natural Science Foundation of China (Grant No. 12171096).


\end{document}